\documentclass[a4paper,10pt]{amsart}

\usepackage[british]{babel}
\usepackage{bbm}
\usepackage{mathrsfs}
\usepackage[twoside=false]{geometry}
\usepackage{setspace}

\newlength{\aufzleft}
\newenvironment{aufz}{\begin{list}{}{\setlength{\listparindent}{0pt}\setlength{\itemsep}{\topsep}\setlength{\labelwidth}{3.2ex}\setlength{\aufzleft}{\labelsep}\addtolength{\aufzleft}{\labelwidth}\setlength{\leftmargin}{\aufzleft}}}{\end{list}}

\newtheoremstyle{par}{1ex}{2ex}{\rm}{}{\bfseries}{}{0.8em}{\thmnumber{(#2)}}
\newtheoremstyle{thm}{1ex}{2ex}{\itshape}{}{\bfseries}{}{0.9em}{\thmnumber{(#2)}\thmname{ #1}\thmnote{ (#3)}}

\theoremstyle{par}
\newtheorem{no}{}[section]

\theoremstyle{thm}
\newtheorem{prop}[no]{Proposition}
\newtheorem{cor}[no]{Corollary}

\newcommand{\dfgl}{\mathrel{\mathop:}=}
\DeclareMathOperator{\spec}{Spec}
\DeclareMathOperator{\Idem}{Idem}
\DeclareMathOperator{\irr}{Irr}

\begin{document}

\title{Irreducibility and integrity of schemes}
\author{Fred Rohrer}
\address{Salamanderweg 7, 7320 Sargans, Switzerland}
\email{fredrohrer@math.ch}

\begin{abstract}
This is a comprehensive study of the relations between the global, local and pointwise variants of irreducibility and integrity of schemes, including examples and counterexamples, and aimed especially at learners of algebraic geometry.
\end{abstract}

\maketitle


\section*{Introduction}

In algebraic geometry, questions of the following form are ubiquitous:
\begin{aufz}
\item[] \textit{Given a property $\mathbf{P}$\,of schemes, what are the relations between having $\mathbf{P}$, having $\mathbf{P}$\,locally, and having $\mathbf{P}$\,pointwise?}
\end{aufz}
(We say that a scheme has $\mathbf{P}$ locally if every point has an open neighbourhood with $\mathbf{P}$, and that it has $\mathbf{P}$ pointwise if every stalk of the structure sheaf has $\mathbf{P}$.) Two basic but important properties of schemes are \textit{irreducibility} and \textit{integrity.} The above question for these two choices of $\mathbf{P}$ seemingly often poses problems, especially for learners of algebraic geometry, but also if one wishes to avoid unnecessary hypotheses like e.g.\ noetherianness.

The aim of this note is to comprehensively answer these two questions without making any claim of originality. None of the results in it are new. In fact, (a variant of) most of them can be found in \cite{ega}. Various of the counterexamples presented here are collected from literature and internet; they are part of folklore, albeit not all of them well-known. Thus, I hope this note will be of use for learners -- and maybe also practitioners -- of algebraic geometry.\smallskip

Concerning terminology we follow Bourbaki's \textit{\'El\'ements de math\'ema\-tique} and Grothendieck's \textit{\'El\'ements de g\'eom\'etrie alg\'ebrique}. By a \textit{space} we always mean a topological space, and by a \textit{ring} we always mean a commutative ring.


\section{Irreducibility of topological spaces}

Irreducibility and local irreducibility of a scheme are properties of its underlying space, and so we begin our study with topological spaces. First, we consider some properties equivalent to or implied by local irreducibility (\ref{1.20}), and then apply this to show that a space is irreducible if and only if it is nonempty, connected and locally irreducible (\ref{1.35}). We exhibit examples, also spectral ones, showing that none of these conditions can be omitted (\ref{1.36}).

\begin{no}\label{1.10}
A) Recall that a space $X$ is called \textit{irreducible} if the intersection of finitely many nonempty open subsets of $X$ is nonempty, or -- equivalently -- if the union of finitely many proper closed subsets of $X$ is a proper subset of $X$. A subset of $X$ is called \textit{irreducible} if it is so furnished with the topology induced by $X$. We denote by $\irr(X)$ the set of irreducible components of $X$, i.e., of $\subseteq$-maximal irreducible subsets of $X$. For basics on irreducibility and irreducible components we refer the reader to \cite[II.4.1]{ac}.

\smallskip

B) Recall that an irreducible component of a space $X$ is connected and hence contained in a connected component of $X$, and that a connected component of $X$ is nonempty and hence contains an irreducible component of $X$. A space is called \textit{totally disconnected} if its connected components are singletons (hence so are its irreducible components).

\smallskip

C) A space $X$ is called \textit{locally irreducible} if every point of $X$ has an irreducible neighbourhood, or -- equivalently, as nonempty open subsets of irreducible spaces are irreducible -- an irreducible open neighbourhood. Irreducible spaces are locally irreducible, but the converse does not hold as exemplified by the empty space (cf.\,\ref{1.35} and \ref{1.36}).

\smallskip

D) Recall that a space is called \textit{spectral} if it is isomorphic to the underlying space of an affine scheme, and \textit{locally spectral} if every point has a spectral open neighbourhood. It is a result of Hochster that a space is locally spectral if and only if it is isomorphic to the underlying space of a scheme (\cite[Theorem 9]{hochster}).
\end{no}

\begin{prop}\label{1.20}
Let $X$ be a space. We consider the following statements.
\begin{aufz}
\item[(1)] $X$ is locally irreducible;
\item[(2)] The irreducible components of $X$ are open;
\item[(3)] $X$ is the sum of its irreducible components;
\item[(4)] The irreducible components of $X$ and the connected components of $X$ coincide;
\item[(5)] The connected components of $X$ are irreducible;
\item[(6)] The irreducible components of $X$ are pairwise disjoint.
\end{aufz}
We have (1)$\Leftrightarrow$(2)$\Leftrightarrow$(3)$\Rightarrow$(4)$\Leftrightarrow$(5)$\Rightarrow$(6), and if\/ $\irr(X)$ is locally finite then (1)--(6) are equivalent.
\end{prop}

\begin{proof}
``(1)$\Rightarrow$(2)'': Suppose $X$ is locally irreducible. Let $x\in Z\in\irr(X)$. Then, $x$ has an irreducible open neighbourhood $U$ in $X$. As $U\cap Z$ is nonempty, it is an irreducible component of the irreducible set $U$ (\cite[II.4.1 Proposition 7]{ac}), hence equals $U$, and thus $U\subseteq Z$. This shows that $Z$ is open.

``(2)$\Rightarrow$(1)'': Clear, as open irreducible components are irreducible neighbourhoods of their points.

``(3)$\Rightarrow$(2)'': Clear from the construction of the sum of spaces.

``(2)$\Rightarrow$(4)'': Suppose the irreducible components of $X$ are open; keep in mind that they are closed and nonempty. A connected component $Y$ of $X$ is contained in every open and closed subset of $X$ meeting $Y$ (\cite[I.11.5]{tg}), thus in every irreducible component of $X$ meeting $Y$. The claim follows now together with \ref{1.10}\,B).

``(2)$\Rightarrow$(3)'': Suppose the irreducible components of $X$ are open. We just saw that they then coincide with the connected components of $X$ which thus are open. It is then clear from the construction of the sum of spaces that $X$ is the sum of its connected components, hence of its irreducible components.

``(4)$\Rightarrow$(5)'': Obvious.

``(5)$\Rightarrow$(4)'': Clear from \ref{1.10}\,B).

``(4)$\Rightarrow$(6)'': Clear, as connected components are pairwise disjoint.

Finally, suppose $\irr(X)$ is locally finite. ``(6)$\Rightarrow$(2)'': Let $x\in Z\in\irr(X)$. Then, $x$ has an open neighbourhood $U$ in $X$ that meets only finitely many elements of $\irr(X)$. The union $Y$ of these finitely many elements of $\irr(X)$ is closed, hence $U\setminus Y$ is an open neighbourhood of $x$ contained in $Z$, and thus $Z$ is open as desired.
\end{proof}

\begin{cor}\label{1.30}
If a space $X$ is locally irreducible then $\irr(X)$ is locally finite.
\end{cor}

\begin{proof}
A point of a locally irreducible space $X$ has an irreducible open neighbourhood $U$. As the irreducible components of $X$ are pairwise disjoint (\ref{1.20}), $U$ is contained in a unique irreducible component of $X$ and in particular meets only finitely many irreducible components. Thus, $\irr(X)$ is locally finite.
\end{proof}

\begin{no}\label{1.40}
A) The implication ``(4)$\Rightarrow$(3)'' in \ref{1.20} does not necessarily hold. Indeed, a totally disconnected space fulfils (4)--(6), and it is locally irreducible if and only if it is discrete. Therefore, every nondiscrete, totally disconnected space (e.g.\,$\mathbbm{Q}$) fulfils (4), but not (3). For a spectral example, see \ref{1.43}.

\smallskip

B) The implication ``(6)$\Rightarrow$(5)'' in \ref{1.20} does not necessarily hold, not even if $X$ is connected. Indeed, a separated space fulfils (6), and it fulfils (5) if and only if it is totally disconnected. Therefore, every separated space that is not totally disconnected fulfils (6), but not (5). In particular, every connected, separated space with at least two points (e.g.\,$\mathbbm{R}$) fulfils (6), but not (5). For a spectral example, see \ref{1.45}.

\smallskip

C) The converse of \ref{1.30} is not true, not even if $X$ is finite: A space with three points, one of them closed and the others open (but not closed), is not locally irreducible, since the closed point has no irreducible neighbourhood. For a spectral example, see \ref{1.36}\,C).
\end{no}

\begin{no}\label{1.37}
In \cite[0.2.1.6]{ega}\footnote{cf.\,\cite[II.4 Exercice 6\,b)]{ac}}, local irreducibility is defined only if $\irr(X)$ is locally finite, and then it is done by means of the (then) equivalent conditions (2), (4), (5) and (6) from \ref{1.20}. In particular, if $\irr(X)$ is locally finite this coincides with our definition by \ref{1.20}. Moreover, our definition does not conflict with the one in \cite{ega} by \ref{1.30}.
\end{no}

\begin{cor}\label{1.35}
A space is irreducible if and only if it is nonempty, connected, and locally irreducible.
\end{cor}

\begin{proof}
A nonempty, connected, and locally irreducible space is its only connected component, and its connected components coinciding with its irreducible components (\ref{1.20}) it follows that it is irreducible. The converse is clear. 
\end{proof}

\begin{no}\label{1.36}
A) None of the conditions in \ref{1.35} can be omitted. Indeed, the empty space is connected and locally irreducible; a discrete space with two points is nonempty, locally irreducible, but not connected; and the space $\mathbbm{R}$ is nonempty, connected, but not irreducible.

\smallskip

B) There are spectral examples as in A). In the first case we can also consider the empty space, and in the second one the spectrum of the product of two fields. In the third case, for a field $K$, the spectrum of $K[X,Y]/\langle XY\rangle$ is nonempty, connected, but not irreducible (as this ring has the two different minimal primes $\langle X\rangle$ and $\langle Y\rangle$).

\smallskip

C) The third example in B) also provides a spectral counterexample to the converse of \ref{1.30} that is connected and has only finitely many irreducible components.
\end{no}

\begin{prop}\label{1.41}
A locally spectral space is totally disconnected if and only if its dimension\footnote{By the dimension of a topological space we always mean its Krull dimension as defined in \cite[VIII.1.1]{ac}.} is at most $0$.
\end{prop}

\begin{proof}
A space has dimension at most $0$ if and only if its points are closed. A totally disconnected space clearly fulfils this. For the converse, keeping in mind that the empty space is totally disconnected, we consider a locally spectral space $X$ with $\dim(X)=0$.

First, we prove that the topology of $X$ has a basis consisting of sets that are open and closed. For this, since $X$ has a spectral open covering, we can suppose $X=\spec(R)$ with a reduced, $0$-dimensional ring $R$. Let $x\in R$. For every $\mathfrak{m}\in\spec(R)$, the local ring $R_{\mathfrak{m}}$ is reduced and $0$-dimensional, thus a field, and therefore $\langle\frac{x}{1}\rangle_{R_{\mathfrak{m}}}=\langle\frac{x^2}{1}\rangle_{R_{\mathfrak{m}}}$. This implies $\langle x\rangle_R=\langle x^2\rangle_R$, hence there exists $y\in R$ with $x=x^2y$. Setting $e\dfgl xy\in R$ it is readily checked that $e$ is idempotent and $\langle x\rangle_R=\langle e\rangle_R$, hence $D(x)=D(e)$. By idempotency, $e(1-e)=0$, implying $V(e)=D(1-e)$, and thus $D(e)$ is open and closed. Therefore, $(D(e))_{e\in\Idem(R)}$ is a basis of the topology of $\spec(R)$ consisting of sets that are open and closed as desired.

Now, back in the general situation, let $x\in X$ and let $Z$ be the connected component of $X$ containing $x$. Then, $$\{x\}\subseteq Z\subseteq\bigcap\{U\subseteq X\mid U\text{ is open and closed }\wedge x\in U\}=$$$$\bigcap\{U\subseteq X\mid U\text{ is open }\wedge x\in U\}=\{x\}.$$ Indeed, the second inclusion holds by \cite[I.11.5]{tg}\footnote{and is in fact, by \cite[8.5]{lazard}, an equality if $X$ is quasicompact and quasiseparated}, the first equality follows from the preceding paragraph, and the second equality holds because every point of $X$ is closed. So, $Z=\{x\}$, and thus $X$ is totally disconnected.
\end{proof}

\begin{cor}\label{1.42}
An infinite spectral space of dimension $0$ is nondiscrete and totally disconnected.
\end{cor}

\begin{proof}
Immediately from \ref{1.41}, for spectral spaces are quasicompact, and quasicompact discrete spaces are finite.
\end{proof}

\begin{no}\label{1.43}
A) There are spectral examples as in \ref{1.40}\,A), i.e., nondiscrete, totally disconnected spectral spaces. Indeed, by \ref{1.42} every infinite affine scheme of dimension $0$ qualifies, and a class of such schemes are spectra of products of infinitely many fields. We provide a direct proof that they fulfil our wishes. 

Let $(K_i)_{i\in I}$ be an infinite family of fields and let $R\dfgl\prod_{i\in I}K_i$. First, if $j\in I$ then $R/\prod_{i\in I\setminus\{j\}}K_i\cong K_j$ is a field, implying $\{\prod_{i\in I\setminus\{j\}}K_i\mid j\in I\}\subseteq\spec(R)$, and therefore $\spec(R)$ is infinite. Furthermore, for $x=(x_i)_{i\in I}\in R$ we define an element $\overline{x}=(\overline{x}_i)_{i\in I}\in R$ with $\overline{x}_i=x_i^{-1}$ if $x_i\neq 0$ and $\overline{x}_i=0$ if $x_i=0$ for $i\in I$. Clearly, if $x\in R$ then $x=x^2\overline{x}$. Let $\mathfrak{p}\in\spec(R)$ and let $x\in R\setminus\mathfrak{p}$. Then, $x=x^2\overline{x}$, implying $x+\mathfrak{p}=(x+\mathfrak{p})(x\overline{x}+\mathfrak{p})$, hence $1+\mathfrak{p}=x\overline{x}+\mathfrak{p}=(x+\mathfrak{p})(\overline{x}+\mathfrak{p})$, so that $x+\mathfrak{p}$ is a unit in $R/\mathfrak{p}$. This shows that $R/\mathfrak{p}$ is a field, thus $\dim(R)=0$. 

\smallskip

B) A further class of spectral examples as in \ref{1.40}\,A) are infinite, compact, totally disconnected spaces; they are spectral by \cite[I.1.3.16]{ega}. An example of such a space is Cantor's ternary set (\cite[IV.2.5 Exemple]{tg}).
\end{no}

\begin{no}\label{1.45}
A) There are spectral examples as in \ref{1.40}\,B). The following (seemingly unpublished) construction by Hochster produces such spaces. More precisely, it yields connected affine schemes that are pointwise integral (hence pointwise irreducible and thus by \ref{2.45}\,a) fulfil property (6) in \ref{1.20}), but not integral (hence not irreducible and thus do not fulfil property (5) in \ref{1.20}) -- cf.\,\ref{3.10}. (A geometric description of a special case of this construction can be found \cite[0568]{stacks}.)

\smallskip

B) The ingredients of the construction are a field $K$ and a gapfree\footnote{An ordered set $(E,\leq)$ is called \textit{gapfree} if it contains two comparable elements and if for $x,y\in E$ with $x<y$ there exists $z\in E$ with $x<z<y$.} totally ordered set $(E,\leq)$, e.g. $\mathbbm{Q}$ with its usual ordering. Let $L\dfgl E\times\mathbbm{N}^*$ as a set. We define a multiplication on $L$ by setting \[(x,m)\cdot(y,n)=\begin{cases}(x,m),&\text{if }x<y;\\(y,n),&\text{if }x>y;\\(x,m+n),&\text{if }x=y.\\\end{cases}\] It is readily checked that this yields a structure of commutative and associative magma on $L$. We denote by $M$ the multiplicatively written commutative monoid obtained from $L$ by adjoining a neutral element (denoted by $1$). Next, we consider the algebra $R=K[M]$ of the monoid $M$ over $K$ and denote by $(e_m)_{m\in M}$ its canonical basis. We will show now that $X=\spec(R)$ has the desired properties.

\smallskip

C) The ring $R$ is not integral, for if $x,y\in E$ with $x<y$ then $(x,1)(1-(y,1))=0$. Let $r\in R\setminus K$, so that $r=r_1e_1+\sum_{\alpha\in L'}r_{\alpha}e_{\alpha}$ with $r_1\in K$, a nonempty finite subset $L'\subseteq L$ and a family $(r_{\alpha})_{\alpha\in L'}$ in $K\setminus 0$. Then, \[l\dfgl\max\{m\in\mathbbm{N}\mid\exists x\in E:(x,m)\in L'\}\text{ and }z\dfgl\max\{x\in E\mid(x,l)\in L'\}\] exist. So, $r$ has a summand of the form $r_{(z,l)}e_{(z,l)}$ with $r_{(z,l)}\in K\setminus 0$, and thus $r^2$ has a summand of the form $r_{(z,l)}^2e_{(z,2l)}$. It follows $r^2\neq 0$, and the choice of $l$ and $z$ implies moreover $r^2\neq r$. Therefore, $R$ is reduced, and $\Idem(R)=\{0,1\}$; thus, $X$ is connected (\cite[II.4.3 Proposition 15 Corollaire 2]{ac}). As it is not integral we see now that $X$ is not locally irreducible. (The fact that $X$ is connected and reduced but not integral could also be deduced from \cite[8.1; 9.9; 10.8]{gilmer}, since the monoid $M$ is readily checked to be torsionfree and aperiodic, but not cancellable\footnote{A multiplicatively written commutative monoid $M$ is called \textit{torsionfree} if for $x,y\in M$ and $n\in\mathbbm{N}$ with $x^n=y^n$ it follows $x=y$, \textit{aperiodic} if for $x\in M$ and $m,n\in\mathbbm{N}$ with $x^m=x^n$ it follows $m=n$, and \textit{cancellable} if for $x,y,z\in M$ with $xz=yz$ it follows $x=y$.}.)

\smallskip

D) Let $\mathfrak{p}\in X$. The sets $I\dfgl\{x\in E\mid e_{(x,1)}\in\mathfrak{p}\}$ and $J\dfgl\{x\in E\mid e_{(x,1)}\notin\mathfrak{p}\}$ define a partition of $E$ such that if $x\in I$ and $y\in J$ then $x<y$. Thus, since $E$ is gapfree, $I$ has no greatest element or $J$ has no smallest element. We consider now the canonical morphism of rings $\eta\colon R\rightarrow R_{\mathfrak{p}}$. If $x\in J$ is not the smallest element of $J$ then there exists $y\in J$ with $y<x$, implying $e_{(y,1)}(1-e_{(x,1)})=0$, and as $e_{(y,1)}\notin\mathfrak{p}$ we get $\eta(e_{(x,1)})=\eta(1)=1_{R_{\mathfrak{p}}}$. If $x\in I$ is not the greatest element of $I$ then there exists $y\in I$ with $x<y$, implying $e_{(x,1)}(1-e_{(y,1)})=0$, and as $e_{(x,1)}\in\mathfrak{p}$ we get $1-e_{(y,1)}\notin\mathfrak{p}$, hence $\eta(e_{(x,1)})=0_{R_{\mathfrak{p}}}$. So, if $I$ has no greatest element and $J$ has no smallest element the above implies that $R_{\mathfrak{p}}=K$, and thus $R_{\mathfrak{p}}$ is integral. Otherwise, if $I$ has a greatest element or $J$ has a smallest element $z$, the above also implies that $R_{\mathfrak{p}}$ is a local ring of fractions of the polynomial algebra $K[e_{(z,1)}]$, and thus in particular integral. Herewith it is proven that $X$ is pointwise integral.

\smallskip

E) The affine scheme $X$ constructed in B) has the property that the stalks of its structure sheaf are discrete valuation rings, and thus it is one-dimensional and pointwise noetherian (but of course nonnoetherian by \ref{1.20}). This follows from the description of $R_{\mathfrak{p}}$ in C) together with \cite[VII.1.6 Th\'eor\`eme 4]{ac}.
\end{no}


\section{Irreducibility of schemes}

Besides the purely topological properties of being irreducible or locally irreducible, the structure of ringed space of a scheme allows a further related notion, namely pointwise irreducibility. We show that a scheme is locally irreducible if and only if it is pointwise irreducible and the set of its irreducible components is locally finite (\ref{2.70}\,a)), and we explain how the spectral counterexamples in the first section show that none of these conditions can be omitted.

\begin{no}
A) A scheme is called \textit{irreducible} or \textit{locally irreducible} if its underlying topological space is so. Irreducible schemes are locally irreducible, but the converse does not hold, not even for affine schemes -- the spectrum of the product of two fields is a counterexample (\ref{1.36}\,B)).

\smallskip

B) A ring is called \textit{irreducible} or \textit{locally irreducible} if its spectrum is so. A ring is irreducible if and only if it has precisely one minimal prime ideal. A scheme $X$ is called \textit{pointwise irreducible} if the ring $\mathscr{O}_{X,x}$ is irreducible for every $x\in X$. A ring is called \textit{pointwise irreducible} if its spectrum is so.
\end{no}

\begin{no}
A) In \cite{ega}, local irreducibility for a scheme $X$ is not defined explicitly, and implicitly only in case $\irr(X)$ is locally finite (cf.~\ref{1.37}). Pointwise irreducibility of a scheme is not defined at all.\smallskip

B) Our definition of irreducibility of a ring conflicts with Bourbaki's definition of irreducibility of ideals in \cite[I.8 Exercise 11]{a}. There, an ideal $\mathfrak{a}$ of a ring $R$ is defined to be irreducible if it is not the intersection of two distinct ideals of $R$ different from $\mathfrak{a}$. As every ring is irreducible in this sense our definition will not cause any confusion.

The relation between these two notions of irreducibility is as follows: An ideal $\mathfrak{a}$ is prime if and only if it is different from $R$, irreducible and equals its own radical. Thus, a ring $R$ is irreducible in our sense if and only if it is nonzero and its nilradical is irreducible in Bourbaki's sense.
\end{no}

\begin{prop}\label{2.40}
A point $x$ of a scheme $X$ is contained in a unique irreducible component of $X$ if and only if $\mathscr{O}_{X,x}$ is irreducible.
\end{prop}

\begin{proof}
We can suppose $X$ is affine with ring $R$. Then, $x\in X$ is contained in two distinct irreducible components of $X$ if and only if it contains two distinct minimal prime ideals of $R$. This is equivalent to $R_x$ having two distinct minimal prime ideals, hence to $\mathscr{O}_{X,x}=R_x$ not being irreducible.
\end{proof}

\begin{cor}\label{2.45}
a) A scheme is pointwise irreducible if and only if its irreducible components are pairwise disjoint.\footnote{This shows that -- for schemes -- pointwise irreducibility is equivalent to property (6) in \ref{1.20}.}

b) Locally irreducible schemes are pointwise irreducible.
\end{cor}

\begin{proof}
Immediately from \ref{2.40} and \ref{1.20}.
\end{proof}

\begin{cor}\label{2.70}
a) A scheme $X$ is locally irreducible if and only if it is pointwise irreducible and $\irr(X)$ is locally finite.

b) A scheme $X$ is irreducible if and only if it is non\-empty, connected and pointwise irreducible and $\irr(X)$ is locally finite.
\end{cor}

\begin{proof}
a) follows from \ref{2.45}, \ref{1.30} and \ref{1.20}; b) follows from \ref{1.35} and a).
\end{proof}

\begin{no}\label{2.75}
A) None of the conditions in \ref{2.70}\,a) can be omitted, not even if $X$ is affine and connected, hence the converse of \ref{2.45}\,b) does not hold. Indeed, in \ref{1.45} we saw a connected affine scheme that is pointwise irreducible but not locally irreducible. Furthermore, for a field $K$, the spectrum of $K[X,Y]/\langle XY\rangle$ is connected but not pointwise irreducible and has only finitely many irreducible components (\ref{1.36}\,B)).

\smallskip

B) The examples in A) and in \ref{1.36}\,B) also show that none of the conditions in \ref{2.70}\,b) can be omitted, not even if $X$ is affine.
\end{no}

\begin{prop}\label{2.60}
If a point $x$ of a scheme $X$ has an irreducible neighbourhood in $X$ then $\mathscr{O}_{X,x}$ is irreducible; the converse holds if $\irr(X)$ is locally finite.
\end{prop}

\begin{proof}
The first statement follows immediately from \ref{2.45}\,b). Conversely, if $\irr(X)$ is locally finite and $\mathscr{O}_{X,x}$ is irreducible, then the union $Z$ of irreducible components of $X$ not containing $x$ is closed (\cite[I.1.5 Proposition 4]{tg}), hence $X\setminus Z$ is an open neighbourhood of $x$. By \ref{2.40} it is contained in the unique irreducible component of $X$ containing $x$ and thus irreducible.
\end{proof}

\begin{no}
In the second statement of \ref{2.60} the hypothesis that $\irr(X)$ is locally finite cannot be omitted, as exemplified by spectra of products of infinitely many fields (\ref{2.75}\,A)).
\end{no}


\section{Integrity of schemes}

Finally, we combine the above results with reducedness (whose global, local and pointwise variants are all equivalent) to obtain similar results on integrity, local integrity and pointwise integrity.

\begin{no}\label{3.10}
A scheme $X$ is called \textit{integral, locally integral,} or \textit{pointwise integral} if it is reduced and moreover irreducible, locally irreducible, or pointwise irreducible, respectively. It is locally integral if and only if every point of $X$ has an integral neighbourhood, or -- equivalently -- an integral open neighbourhood; it is pointwise integral if the ring $\mathscr{O}_{X,x}$ is integral for every $x\in X$. Integral schemes are locally integral, and locally integral schemes are pointwise integral (\ref{2.45}\,b)).
\end{no}

\begin{no}
In \cite[I.2.1.8]{ega} a scheme is defined to be integral or locally integral as above. In \cite[IV.4.6.9]{ega} a scheme of finite type over a field is (implicitly) defined to be pointwise integral as above, but no definition for more general schemes appears in \cite{ega}.
\end{no}

\begin{prop}\label{3.20}
A scheme $X$ is irreducible, locally irreducible, or pointwise irreducible if and only if $X_{\mathrm{red}}$ is integral, locally integral, or pointwise integral, respectively; moreover, $\irr(X)=\irr(X_{\mathrm{red}})$.
\end{prop}

\begin{proof}
The underlying topological spaces of $X$ and $X_{\mathrm{red}}$ are equal, hence they have the same irreducible components, and $X_{\mathrm{red}}$ is reduced. Therefore, $X$ is (locally) irreducible if and only if $X_{\mathrm{red}}$ is (locally) integral. If $x\in X$ then $\mathscr{O}_{X_{\mathrm{red}},x}=(\mathscr{O}_{X,x})_{\mathrm{red}}$ is reduced, hence $X$ is pointwise irreducible if and only if $X_{\mathrm{red}}$ is pointwise integral.
\end{proof}

\begin{prop}\label{3.30}
Let $X$ be a scheme.

a) The following statements are equivalent: (i) $X$ is locally integral; (ii) $X$ is pointwise integral and\/ $\irr(X)$ is locally finite.

b) The following statements are equivalent: (i) $X$ is integral; (ii) $X$ is nonempty, connected and locally integral; (iii) $X$ is nonempty, connected and pointwise integral and\/ $\irr(X)$ is locally finite.
\end{prop}

\begin{proof}
Immediately from \ref{2.70}.
\end{proof}

\begin{no}
It follows from \ref{3.20} and \ref{2.75} that none of the conditions in \ref{3.30} can be omitted.
\end{no}


\smallskip\noindent\textbf{Acknowledgement:} I thank the referee for his careful reading and his remarks.


\end{document}